\documentclass[11pt]{amsart}
\usepackage{amsmath,amssymb,amsthm}
\usepackage[margin=1in]{geometry}
\usepackage{cite}
\usepackage[hidelinks]{hyperref}

\theoremstyle{plain}
\newtheorem{theorem}{Theorem}[section]
\newtheorem{lemma}[theorem]{Lemma}
\newtheorem{proposition}[theorem]{Proposition}
\theoremstyle{definition}
\newtheorem{definition}[theorem]{Definition}
\theoremstyle{remark}
\newtheorem{remark}[theorem]{Remark}
\numberwithin{equation}{section}

\title[Gaussian bounds for distinct subset sums]{Gaussian Lower Bounds for Distinct Subset Sums in Higher Dimensions}
\author[Zijie Gu]{Zijie Gu}
\dedicatory{\normalfont\parbox{0.95\textwidth}{\centering\small
Independent researcher, Shanghai, China\\[3pt]
\href{mailto:hilightgu@gmail.com}{hilightgu@gmail.com}\\[3pt]
}}
\hypersetup{
  pdftitle={Gaussian Lower Bounds for Distinct Subset Sums in Higher Dimensions},
  pdfauthor={Zijie Gu},
  pdfkeywords={distinct subset sums, additive combinatorics, Gaussian bounds, probabilistic method, lattice sums, quadratic forms}
}

\begin{document}
\begin{abstract}
For sequences of $n$ nonnegative integer vectors in $\mathbb Z^k$ with distinct subset sums, we prove, for fixed $k$, the uniform lower bound $M\ge (C_k-o(1))2^{n/k}/\sqrt n$ for the largest coordinate, where
\[
C_k=\sqrt{\frac{2k}{\pi}}\,
\frac{(k+1)^{(k-1)/(2k)}}{\sqrt{\lfloor (k+1)^2/4\rfloor}}.
\]
The proof compares a Gaussian-kernel expectation with a sum over a translated lattice, uniformly over the vector sequence. We optimise the resulting quadratic-form constant using a vertex second-moment certificate. We show that $C_k$ also exceeds the variance constant obtained by optimising an ellipsoidal geometry.
\end{abstract}
\maketitle

\section{Introduction}\label{sec:introduction}

This paper seeks lower bounds for the largest coordinate of a sequence of nonnegative integer vectors whose subset sums are all distinct. In dimension one, this is the Erd\H{o}s distinct subset sums problem. For a positive integer $n$ and positive integers $0<a_1<\cdots<a_n$, the condition is
\[
\sum_{i\in S}a_i=\sum_{i\in T}a_i\quad\Longrightarrow\quad S=T,
\qquad S,T\subseteq\{1,\ldots,n\}.
\]
The variance method gives the classical Erd\H{o}s--Moser bound $a_n\ge 2^n/(4\sqrt n)$, as recalled in the opening discussion of Dubroff, Fox, and Xu \cite{dubroff2021note}. In the same note \cite[Theorem~1]{dubroff2021note}, they give two proofs of the stronger asymptotic bound
\[
a_n\ge \left(\sqrt{\frac2\pi}-o(1)\right)\frac{2^n}{\sqrt n}.
\]
The first bound holds for every $n\ge1$; the second is asymptotic as $n\to\infty$. Dubroff, Fox, and Xu explicitly record that the latter constant had already been obtained in unpublished work of Elkies and Gleason. These scalar results apply to our setting when $k=1$, after ordering the entries increasingly.

In 2026, Costa and Della Fiore \cite[Theorem~1.1 and Corollary~5.3]{costa2026bandwidth} prove the finite scalar bound
\[
a_n\ge\sum_{j=0}^{n-1}\binom{j}{\lfloor j/2\rfloor}\qquad(n\ge1).
\]
For $n\ge3$, this improves the bound obtained in the second proof of Dubroff, Fox, and Xu, while retaining the leading constant $\sqrt{2/\pi}$.

For $n\ge1$, let $m_1(n)$ denote the minimum attainable largest entry $a_n$ among the $n$-term strictly increasing positive integer sequences with distinct subset sums described above.

In dimension one, the classical Erd\H{o}s conjecture asks whether $m_1(n)\ge c2^n$ for a positive absolute constant $c$. In September of 2026, Tom Adamczewski and Thomas F. Bloom report a disproof in \emph{FrontierMath Erd\H{o}s}, \cite[Appendix~B.1, Theorem~1]{adamczewski2026frontiermath}. More precisely, they report that for every real $\varepsilon>0$ and every integer $L\ge1$, there exist integers $n\ge L$ and $N\ge1$ and a set $\mathcal A\subseteq\{1,\ldots,N\}$ with $|\mathcal A|=n$ and all subset sums distinct, such that $N\le\varepsilon2^n$. Conditional on this reported theorem, no positive absolute $c$ satisfies the classical conjecture. 

Steinerberger \cite[Theorem~1 and Corollary~2]{steinerberger2023some} obtains the same scalar asymptotic constant through a Fourier integral involving a product of squared cosines. He also compares a smoothed signed-sum distribution with a Gaussian density, under an additional bound on the largest summand \cite[Theorem~2]{steinerberger2023some}. This is a methodological predecessor to the present argument.

For the vector problem, we use the following conventions. Throughout, $k$ and $n$ are positive integers, with $k$ fixed as $n\to\infty$.

\begin{definition}\label{def:setting}
For positive integers $k,n$, a sequence $\Sigma=(a_1,\ldots,a_n)$ with $a_i\in\mathbb Z_{\ge0}^k$ is \emph{admissible} if its $2^n$ subset sums are all distinct. For such a sequence, write
\[
M(\Sigma)=\max_{1\le i\le n,\,1\le j\le k}(a_i)_j.
\]
Here $(a_i)_j$ denotes the $j$th coordinate of $a_i$. Define
\[
m_k(n)=\min\{M(\Sigma):\Sigma\text{ is an admissible sequence of length }n\text{ in }\mathbb Z_{\ge0}^k\}.
\]
\end{definition}

Unless otherwise specified, write $M=M(\Sigma)\ge1$.

The minimum exists for every $k,n\ge1$. Distribute the $n$ vectors as evenly as possible among the $k$ coordinate axes, taking $e_j,2e_j,4e_j,\ldots$ on each used axis, where $e_j$ is the $j$th standard basis vector. Uniqueness of binary representation makes the subset sums distinct, so the set defining $m_k(n)$ is nonempty. Its integer values therefore have an attained minimum, and the construction gives
\[
m_k(n)\le2^{\lceil n/k\rceil-1}.
\]
This also applies when $n<k$, with some axes unused.

For each fixed positive integer $k$, Costa, Dalai, and Della Fiore \cite[arXiv:2107.07885v3, Theorem~2.5]{costa2021variations} obtain the following asymptotic bound for every admissible sequence. We take $\lambda=1$ in their theorem: their family $\mathcal F_{1,n}$ is the full power set, and their $M$-bounded integer vectors lie in $[0,M]^k$, exactly as in Definition~\ref{def:setting}. With $\Gamma$ denoting Euler's gamma function, their bound is
\begin{equation}\label{eq:variance-bound}
M\ge\bigl(D_{k,2}-o(1)\bigr)\frac{2^{n/k}}{\sqrt n},
\qquad D_{k,2}=\frac{2\Gamma(k/2+1)^{1/k}}{\sqrt{\pi(k+2)}}.
\end{equation}

The author's earlier preprint \cite{gu2025generalisation} studied the same
vector problem using first- and third-moment methods.

We use a Gaussian kernel in place of a moment and prove the estimate directly for bounded vector arrays, uniformly even when their covariance matrices vary with $n$ or are singular. The coordinate bound and the distinctness condition give, respectively, lower and upper bounds on the same expectation. A suitable quadratic form then gives a single constant for every fixed dimension.

\begin{theorem}\label{mainthe}
Fix a positive integer $k$. Let $\Sigma=(a_1,\ldots,a_n)$ be an admissible sequence as in Definition~\ref{def:setting}, and let $M=M(\Sigma)$. As $n\to\infty$,
\begin{equation}\label{eq:main-bound}
M\ge\bigl(C_k-o(1)\bigr)\frac{2^{n/k}}{\sqrt n},
\qquad C_k=\sqrt{\frac{2k}{\pi}}\,
\frac{(k+1)^{(k-1)/(2k)}}{\sqrt{\lfloor(k+1)^2/4\rfloor}}.
\end{equation}
The error depends only on $k,n$ and is uniform over all such sequences. In particular, the same bound holds for $m_k(n)$.
\end{theorem}

For example,
\[
C_1=\sqrt{\frac2\pi},\qquad C_2=\sqrt{\frac{2\sqrt3}{\pi}}.
\]
Thus the one-dimensional constant agrees with the scalar bound above. For every fixed $k$, the strict inequality $C_k>D_{k,2}$ improves the vector variance benchmark \eqref{eq:variance-bound}.

Theorem~\ref{thm:quadratic-optimality} then proves that the chosen quadratic form maximises the explicit formula $C(A)$ defined in Proposition~\ref{prop:quadratic-bound}. Its vertex certificate also describes a minimum-volume enclosing ellipsoid of a classical projected cube, as explained in Appendix~\ref{app:geometry}; the theorem does not determine the optimal constant for $m_k(n)$.

To separate the effect of the quadratic form from that of the Gaussian kernel, Proposition~\ref{prop:ellipsoidal-variance} also establishes the optimised variance constant
\[
V_k=D_{k,2}\frac{C_k}{C_1}<C_k.
\]

\section{Gaussian Comparison Framework}\label{sec:gaussian-estimates}

Bounded coefficients give a lower bound for a Gaussian-kernel expectation; distinctness and integrality give an upper bound via a translated lattice. For fixed kernel parameters, these estimates are uniform over admissible sequences, including those with singular covariance matrices, and together bound the largest coordinate from below.

\subsection{Centred Subset Sums and Normalisation}\label{sec:model}

Let $\Sigma=(a_1,\ldots,a_n)$ be admissible as in Definition~\ref{def:setting}, with $M=M(\Sigma)$. Let $r_1,\ldots,r_n$ be independent and uniformly distributed on $\{-1,1\}$, and set
\[
X=\frac12\sum_{i=1}^n r_i a_i.
\]
If $S=\{i:r_i=1\}$, then
\begin{equation}\label{eq:centered-subset}
X=\sum_{i\in S}a_i-\frac12\sum_{i=1}^n a_i.
\end{equation}
Consequently, $X$ has $2^n$ distinct values, each with probability $2^{-n}$. All values lie in the same translate of $\mathbb Z^k$. The coordinates of $X$ need not be independent: they use the same random signs. Our arguments require independence of the signs only.

To apply the bounded-array Gaussian estimate and track the lattice spacing, define
\begin{equation}\label{eq:normalised-sum}
u_i=\frac{a_i}{M},\qquad
Y=\frac{2X}{M\sqrt n}=\frac1{\sqrt n}\sum_{i=1}^n r_i u_i.
\end{equation}
Then $u_i\in[0,1]^k$, and the values of $Y$ lie in
\begin{equation}\label{eq:lattice-spacing}
h(\mathbb Z^k+\theta),\qquad h=\frac2{M\sqrt n},
\end{equation}
for some $\theta\in\mathbb R^k$ depending on the sequence.

For this normalised vector $Y$, fix a real symmetric positive definite $k\times k$ matrix $A$ and a parameter $t>0$, both independent of $n$. We compare two bounds on
\[
Q_t=\mathbb E e^{-tY^{\mathsf T}AY}.
\]
\subsection{A Gaussian-Kernel Lower Bound for Bounded Random Sums}

\begin{lemma}\label{lem:gaussian-probability}
Let $k,n\ge1$ be integers and let $b>0$. Let $z_1,\ldots,z_n\in\mathbb R^k$ satisfy $\|z_i\|_2^2\le b$. Let $r_1,\ldots,r_n$ be independent and uniformly distributed on $\{-1,1\}$, and put
\[
Z=\frac1{\sqrt n}\sum_{i=1}^n r_i z_i.
\]
For fixed $k,b,t$, with $t>0$, as $n\to\infty$,
\begin{equation}\label{eq:probability-bound}
\mathbb E e^{-t\|Z\|_2^2}
\ge\left(1+\frac{2tb}{k}\right)^{-k/2}-O_{k,b,t}(n^{-1}).
\end{equation}
The implied constant does not depend on the vectors $z_i$, which may vary with $n$ subject to the stated bound. The covariance matrix of $Z$ may be singular, and the coordinates of $Z$ need not be independent.
\end{lemma}

\begin{proof}
Let $G$ be a standard Gaussian vector in $\mathbb R^k$, independent of the signs. For every fixed $x\in\mathbb R^k$,
\begin{equation}\label{eq:gaussian-cosine}
\mathbb E_G\cos(\sqrt{2t}\,G\cdot x)=e^{-t\|x\|_2^2}.
\end{equation}
The one-dimensional Gaussian characteristic function and independence of the Gaussian coordinates give this identity. Applying \eqref{eq:gaussian-cosine} to each value of $Z$, we may interchange the sign and Gaussian expectations because the cosine integrand is bounded. Conditional on $G$, independence and symmetry of the signs factor the cosine average, giving
\begin{equation}\label{eq:cosine-product}
\mathbb E e^{-t\|Z\|_2^2}
=\mathbb E_G\prod_{i=1}^n\cos\left(\sqrt{\frac{2t}{n}}\,G\cdot z_i\right).
\end{equation}

Define the positive semidefinite matrix
\[
T=\frac1n\sum_{i=1}^n z_i z_i^{\mathsf T}.
\]
Taking traces and using the assumed norm bound gives
\[
\operatorname{tr}T
=\frac1n\sum_{i=1}^n\|z_i\|_2^2\le b.
\]
The function $\log\cos x$ is smooth on $[-1/2,1/2]$, with bounded fourth derivative. Taylor's theorem therefore gives an absolute constant $c>0$ such that
\[
\log\cos x\ge-\frac{x^2}{2}-cx^4\qquad (|x|\le1/2).
\]
To put all cosine arguments within this Taylor interval, define the event in the probability space of the auxiliary Gaussian vector $G$ by
\[
E_n=\left\{\|G\|_2\le\frac{\sqrt n}{2\sqrt{2tb}}\right\}.
\]
Cauchy--Schwarz and $\|z_i\|_2\le\sqrt b$ give, on $E_n$, for every $i$,
\[
\left|\sqrt{\frac{2t}{n}}\,G\cdot z_i\right|
\le\sqrt{\frac{2t}{n}}\,\|G\|_2\sqrt b\le\frac12.
\]
Thus the Taylor estimate applies simultaneously to every factor in \eqref{eq:cosine-product}. Moreover,
\begin{align*}
\frac12\sum_{i=1}^n\left(\sqrt{\frac{2t}{n}}\,G\cdot z_i\right)^2
&=tG^{\mathsf T}TG,\\
\sum_{i=1}^n\left(\sqrt{\frac{2t}{n}}\,G\cdot z_i\right)^4
&\le\frac{4t^2b^2}{n}\|G\|_2^4.
\end{align*}
Exponentiating the sum of the Taylor bounds gives, on $E_n$,
\[
\prod_{i=1}^n\cos\left(\sqrt{\frac{2t}{n}}\,G\cdot z_i\right)
\ge e^{-tG^{\mathsf T}TG}
\exp\left(-\frac{4ct^2b^2}{n}\|G\|_2^4\right).
\]
Use $e^{-x}\ge1-x$ for $x\ge0$. Since $T\succeq0$, we also have $0<e^{-tG^{\mathsf T}TG}\le1$, so the multiplicative error gives the additive bound
\[
\prod_{i=1}^n\cos\left(\sqrt{\frac{2t}{n}}\,G\cdot z_i\right)
\ge e^{-tG^{\mathsf T}TG}-\frac{4ct^2b^2}{n}\|G\|_2^4.
\]
Outside $E_n$, the cosine product lies in $[-1,1]$ and the Gaussian term lies in $(0,1]$, so their difference is at least $-2$. Since $\mathbb E\|G\|_2^4=k(k+2)$, Markov's inequality gives
\[
\mathbb P(E_n^{\mathrm c})\le\frac{64t^2b^2k(k+2)}{n^2}.
\]
Taking expectations over $E_n$ and its complement therefore gives the explicit estimate
\begin{align*}
\mathbb E e^{-t\|Z\|_2^2}
&\ge\mathbb E_G e^{-tG^{\mathsf T}TG}
-\frac{4ct^2b^2}{n}\mathbb E\|G\|_2^4-2\mathbb P(E_n^{\mathrm c})\\
&\ge\mathbb E_G e^{-tG^{\mathsf T}TG}
-\frac{4ct^2b^2k(k+2)}n
-\frac{128t^2b^2k(k+2)}{n^2}.
\end{align*}
Both error terms depend only on $k,b,t,n$, proving uniformity over the vectors. In particular,
\begin{equation}\label{eq:gaussian-comparison}
\mathbb E e^{-t\|Z\|_2^2}
\ge\mathbb E_G e^{-tG^{\mathsf T}TG}-O_{k,b,t}(n^{-1}).
\end{equation}

Let $\lambda_1,\ldots,\lambda_k\ge0$ be the eigenvalues of $T$. Orthogonal diagonalisation and Gaussian integration give
\[
\mathbb E_G e^{-tG^{\mathsf T}TG}
=\prod_{j=1}^k(1+2t\lambda_j)^{-1/2}.
\]
Every factor $1+2t\lambda_j$ is positive, including when $\lambda_j=0$. The first inequality below is the arithmetic--geometric mean inequality for these factors; the second uses $\operatorname{tr}T\le b$:
\[
\prod_{j=1}^k(1+2t\lambda_j)
\le\left(1+\frac{2t\operatorname{tr}T}{k}\right)^k
\le\left(1+\frac{2tb}{k}\right)^k.
\]
Taking reciprocal square roots reverses this inequality, so the Gaussian expectation is at least $(1+2tb/k)^{-k/2}$. Combining this with \eqref{eq:gaussian-comparison} proves the lemma, without requiring $T$ to be invertible.
\end{proof}

\subsection{A Uniform Upper Bound for Gaussian Lattice Sums}

We need a lattice-sum estimate uniform in the translation. We average over a cube centred at each lattice point; symmetry removes the linear term in the quadratic form.

\begin{lemma}\label{lem:gaussian-lattice}
Let $k\ge1$ be an integer and let $A$ be a real symmetric positive definite $k\times k$ matrix. For every $t,h>0$ and $\theta\in\mathbb R^k$,
\begin{equation}\label{eq:lattice-bound}
\sum_{z\in\mathbb Z^k}e^{-t[h(z+\theta)]^{\mathsf T}A[h(z+\theta)]}
\le e^{th^2\operatorname{tr}A/12}\,h^{-k}
\frac{(\pi/t)^{k/2}}{\sqrt{\det A}}.
\end{equation}
For fixed $A,t$, the exponential factor is $1+O_{A,t}(h^2)$ as $h\to0$, uniformly in $\theta$.
\end{lemma}

\begin{proof}
Write $f(x)=e^{-tx^{\mathsf T}Ax}$. For each lattice point $\xi\in h(\mathbb Z^k+\theta)$, let
\[
\mathcal C_\xi=\xi+[-h/2,h/2)^k.
\]
These cubes partition $\mathbb R^k$ and each has volume $h^k$. Let $U$ be uniform on $[-h/2,h/2)^k$. Its coordinates are independent, with
\[
\mathbb E U_j=0,\qquad \mathbb E U_j^2=h^2/12,\qquad
\mathbb E[U_jU_\ell]=0\quad(j\ne\ell).
\]
Consequently,
\begin{equation}\label{eq:symmetric-average}
\mathbb E[(\xi+U)^{\mathsf T}A(\xi+U)]
=\xi^{\mathsf T}A\xi+\frac{h^2}{12}\operatorname{tr}A.
\end{equation}
Applying Jensen's inequality to the convex scalar function $s\mapsto e^{-ts}$ and using the $\xi$-independent correction in \eqref{eq:symmetric-average} gives
\[
\frac1{h^k}\int_{\mathcal C_\xi}f(x)\,dx
=\mathbb E f(\xi+U)\ge f(\xi)e^{-th^2\operatorname{tr}A/12}.
\]
Thus
\begin{equation}\label{eq:cube-comparison}
h^k f(\xi)\le e^{th^2\operatorname{tr}A/12}\int_{\mathcal C_\xi}f(x)\,dx.
\end{equation}
Summing over the partition, using nonnegativity, yields
\[
h^k\sum_{\xi\in h(\mathbb Z^k+\theta)}f(\xi)
\le e^{th^2\operatorname{tr}A/12}\int_{\mathbb R^k}e^{-tx^{\mathsf T}Ax}\,dx.
\]
The substitution $y=\sqrt t\,A^{1/2}x$ gives
\begin{align*}
\int_{\mathbb R^k}e^{-tx^{\mathsf T}Ax}\,dx
&=t^{-k/2}(\det A)^{-1/2}\int_{\mathbb R^k}e^{-\|y\|_2^2}\,dy\\
&=\frac{(\pi/t)^{k/2}}{\sqrt{\det A}}.
\end{align*}
This proves \eqref{eq:lattice-bound}; its factor and all preceding estimates are independent of the translation.
\end{proof}

For an admissible sequence, let $y_1,\ldots,y_{2^n}$ be the distinct, equiprobable values of $Y$ in \eqref{eq:lattice-spacing}. Since the Gaussian weights are nonnegative,
\[
Q_t=\frac1{2^n}\sum_{j=1}^{2^n}e^{-ty_j^{\mathsf T}Ay_j}
\le\frac1{2^n}\sum_{z\in\mathbb Z^k}
 e^{-t[h(z+\theta)]^{\mathsf T}A[h(z+\theta)]}.
\]
Lemma~\ref{lem:gaussian-lattice}, with $h=2/(M\sqrt n)$, gives
\begin{equation}\label{eq:geometric-bound}
Q_t\le e^{t\operatorname{tr}A/(3M^2n)}\,
\frac1{2^n}\left(\frac{M\sqrt n}{2}\right)^k
\frac{(\pi/t)^{k/2}}{\sqrt{\det A}}.
\end{equation}
The bound \eqref{eq:geometric-bound} holds for every $n\ge1$, every symmetric positive definite $A$, and every $t>0$. For fixed $A,t$, the exponential factor is $1+O_{A,t}(n^{-1})$, uniformly over the sequence, because $M\ge1$.

\section{Main Results}\label{sec:main}

Proposition~\ref{prop:quadratic-bound} gives the bound for a general quadratic form; evaluating a particular form then proves Theorem~\ref{mainthe}. Subsection~\ref{sec:optimality} separately certifies optimality within the quadratic-form formula.

\subsection{The Bound Associated with a General Quadratic Form}

\begin{proposition}\label{prop:quadratic-bound}
Fix a positive integer $k$ and a real symmetric positive definite $k\times k$ matrix $A$, independent of $n$. Let $\Sigma=(a_1,\ldots,a_n)$ be admissible as in Definition~\ref{def:setting}, with $M=M(\Sigma)$. Define
\[
b_A=\max_{u\in[0,1]^k}u^{\mathsf T}Au,\qquad
C(A)=\sqrt{\frac{2k}{\pi}}\,\frac{(\det A)^{1/(2k)}}{\sqrt{b_A}}.
\]
As $n\to\infty$,
\[
M\ge\bigl(C(A)-o(1)\bigr)\frac{2^{n/k}}{\sqrt n},
\]
with an error depending only on $A,k,n$, uniformly over the admissible sequence. The same bound holds for $m_k(n)$.
\end{proposition}

\begin{proof}
Use $u_i,Y$ from \eqref{eq:normalised-sum} and set $z_i=A^{1/2}u_i$. Positive definiteness gives $b_A\ge e_1^{\mathsf T}Ae_1>0$, while $\|z_i\|_2^2=u_i^{\mathsf T}Au_i\le b_A$. Also,
\[
Y^{\mathsf T}AY=\left\|\frac1{\sqrt n}\sum_{i=1}^n r_i z_i\right\|_2^2.
\]
Lemma~\ref{lem:gaussian-probability} therefore gives, for fixed $t>0$,
\begin{equation}\label{eq:quadratic-probability}
Q_t\ge\left(1+\frac{2tb_A}{k}\right)^{-k/2}-O_{A,k,t}(n^{-1}).
\end{equation}
Combining this with \eqref{eq:geometric-bound} and writing $R_n=M\sqrt n/2^{n/k}$ yields
\[
\left(1+\frac{2tb_A}{k}\right)^{-k/2}-O_{A,k,t}(n^{-1})
\le\bigl(1+O_{A,t}(n^{-1})\bigr)
\left(\frac{R_n}{2}\right)^k\frac{(\pi/t)^{k/2}}{\sqrt{\det A}}.
\]
The leading term on the left is positive for fixed $t$. Thus there are constants $\alpha_t,\beta_t\ge0$, depending only on $A,k,t$, such that, for all sufficiently large $n$ and every admissible sequence,
\[
R_n\ge F_A(t)
\left(\frac{1-\alpha_t/n}{1+\beta_t/n}\right)^{1/k},
\qquad
F_A(t)=2\sqrt{\frac t\pi}\,(\det A)^{1/(2k)}
\left(1+\frac{2tb_A}{k}\right)^{-1/2}.
\]
The numerator is positive for all sufficiently large $n$, with a threshold independent of the sequence. Moreover,
\[
\lim_{t\to\infty}F_A(t)
=\sqrt{\frac{2k}{\pi}}\,\frac{(\det A)^{1/(2k)}}{\sqrt{b_A}}
=C(A).
\]
To obtain the constant $C(A)$ while retaining the fixed-$t$ error estimate, let $\delta>0$. First choose a finite $t=t(A,k,\delta)$ such that $F_A(t)\ge C(A)-\delta/2$. Keep this $t$ fixed and choose $n_0=n_0(A,k,\delta)$ so that the factor in the finite-$n$ bound changes $F_A(t)$ by at most $\delta/2$ whenever $n\ge n_0$. Then
\[
M\ge\bigl(C(A)-\delta\bigr)\frac{2^{n/k}}{\sqrt n}
\quad\text{for every admissible sequence and every }n\ge n_0.
\]
This proves the asserted uniform error.
\end{proof}

The finite-$n$ Gaussian and lattice estimates above are used with fixed $A,t$; no uniformity for arbitrary choices $A=A_n$ or $t=t_n$ is asserted.

\subsection{Choice of Quadratic Form and Proof of the Main Theorem}

Proposition~\ref{prop:quadratic-bound} reduces the choice of the leading constant to maximising $(\det A)^{1/(2k)}/\sqrt{b_A}$. A useful candidate is invariant under coordinate permutations and scales the all-ones direction differently from its orthogonal complement. Let $I_k$ denote the identity matrix, let $\mathbf1=(1,\ldots,1)^{\mathsf T}$, and take
\begin{equation}\label{eq:H}
H=I_k-\frac{\mathbf1\mathbf1^{\mathsf T}}{k+1}.
\end{equation}

\begin{proposition}\label{prop:H}
For every positive integer $k$, the matrix $H$ in \eqref{eq:H} is positive definite and satisfies
\[
\det H=\frac1{k+1},\qquad b_H=\frac{\lfloor(k+1)^2/4\rfloor}{k+1}.
\]
\end{proposition}

\begin{proof}
The matrix $\mathbf1\mathbf1^{\mathsf T}$ has eigenvalue $k$ in the $\mathbf1$ direction and eigenvalue $0$ on its orthogonal complement. Thus the eigenvalues of $H$ are $1/(k+1),1,\ldots,1$, proving positive definiteness and the determinant formula.

The quadratic form
\[
q_H(u)=\sum_{j=1}^k u_j^2-\frac1{k+1}\left(\sum_{j=1}^k u_j\right)^2
\]
is convex in each coordinate separately. Holding the other coordinates fixed, a maximum on $[0,1]$ can be attained at an endpoint. Repeating this for every coordinate shows that the maximum on the cube is attained at a vertex. At a vertex with $r$ coordinates equal to $1$,
\[
q_H(u)=r-\frac{r^2}{k+1}=\frac{r(k+1-r)}{k+1}.
\]
The maximum of $r(k+1-r)$ over integers $0\le r\le k$ is $\lfloor(k+1)^2/4\rfloor$, giving the formula for $b_H$.
\end{proof}

\begin{proof}[Proof of Theorem~\ref{mainthe}]
Take $A=H$ in Proposition~\ref{prop:quadratic-bound}. Using Proposition~\ref{prop:H},
\[
C(H)=\sqrt{\frac{2k}{\pi}}\,
\frac{(k+1)^{-1/(2k)}}{\sqrt{\lfloor(k+1)^2/4\rfloor/(k+1)}}
=C_k.
\]
This proves \eqref{eq:main-bound}.
\end{proof}

\subsection{Optimality of the Quadratic-Form Formula}\label{sec:optimality}

\begin{theorem}\label{thm:quadratic-optimality}
Fix a positive integer $k$. Let $C(A)$ be defined by Proposition~\ref{prop:quadratic-bound} and let $H$ be given by \eqref{eq:H}. Then
\[
\sup_{A\succ0}C(A)=C(H)=C_k,
\]
where the supremum ranges over all real symmetric positive definite $k\times k$ matrices. This optimises the formula $C(A)$, not the lower bound for $m_k(n)$ over all possible methods.
\end{theorem}

\begin{proof}
For $k=1$, $A$ is a positive scalar and $b_A=A$, so $C(A)=\sqrt{2/\pi}=C_1$. Suppose henceforth that $k\ge2$. Recall from Proposition~\ref{prop:H} that
\[
b_H=\frac{\lfloor(k+1)^2/4\rfloor}{k+1}.
\]
We construct a random cube vertex $V\in\{0,1\}^k$ whose uncentred second moment is proportional to $H^{-1}$. This choice will make the averaging and matrix inequalities below sharp at $A=H$. Specifically, we require
\begin{equation}\label{eq:vertex-second-moment}
S:=\mathbb E[VV^{\mathsf T}]=\frac{b_H}{k}\left(I_k+\mathbf1\mathbf1^{\mathsf T}\right).
\end{equation}
First choose the number $R$ of coordinates equal to $1$, and then choose uniformly among the vertices with that number of ones. If $k$ is odd, set $R=(k+1)/2$. Here $b_H=(k+1)/4$, and
\[
\mathbb E R=2b_H,\qquad \mathbb E[R(R-1)]=(k-1)b_H.
\]
If $k=2m$ is even, set
\[
\mathbb P(R=m)=\frac{m+1}{2m+1},\qquad \mathbb P(R=m+1)=\frac m{2m+1}.
\]
Since $b_H=m(m+1)/(2m+1)$, direct calculation gives
\[
\mathbb E R=\frac{2m(m+1)}{2m+1}=2b_H,\qquad
\mathbb E[R(R-1)]=\frac{m(m+1)(2m-1)}{2m+1}=(k-1)b_H.
\]
Coordinate symmetry now gives, for $j\ne\ell$,
\[
\mathbb E V_j^2=\frac{\mathbb E R}{k}=\frac{2b_H}{k},\qquad
\mathbb E[V_jV_\ell]=\frac{\mathbb E[R(R-1)]}{k(k-1)}=\frac{b_H}{k}.
\]
This proves \eqref{eq:vertex-second-moment}.

The matrix $S$ is positive definite and satisfies
\[
H^{-1}=I_k+\mathbf1\mathbf1^{\mathsf T},\qquad S=\frac{b_H}{k}H^{-1}.
\]

For every positive definite $A$, the support of $V$ lies in $[0,1]^k$, so averaging its quadratic-form values gives the first inequality below. For the second, apply the arithmetic--geometric mean inequality to the positive eigenvalues of the symmetric positive definite matrix $A^{1/2}SA^{1/2}$. Its trace is $\operatorname{tr}(AS)$ and its determinant is $\det A\det S$; no commutativity of $A$ and $S$ is needed. Thus
\[
b_A\ge\mathbb E[V^{\mathsf T}AV]=\operatorname{tr}(AS)
\ge k(\det A\det S)^{1/k}.
\]
Consequently,
\[
C(A)\le\sqrt{\frac2\pi}\,(\det S)^{-1/(2k)}.
\]
By Proposition~\ref{prop:H} and $S=(b_H/k)H^{-1}$,
\[
\det S=\frac{(b_H/k)^k}{\det H}=\left(\frac{b_H}{k}\right)^k(k+1).
\]
Substitution gives
\[
C(A)\le\sqrt{\frac{2k}{\pi b_H}}\,(k+1)^{-1/(2k)}=C_k.
\]
The chosen values of $R$ all maximise $R(k+1-R)$, so $V^{\mathsf T}HV=b_H$ throughout the support of $V$. Also, $H^{1/2}SH^{1/2}=(b_H/k)I_k$. Thus both inequalities in $b_H\ge\operatorname{tr}(HS)\ge k(\det H\det S)^{1/k}$ are equalities, and $H$ attains the upper bound $C_k$.
\end{proof}

The geometric interpretation of this certificate is given in Appendix~\ref{app:geometry}, Remark~\ref{rem:ellipsoid-context}.

\section{Optimised Variance Comparison}\label{sec:variance}

The variance bound \eqref{eq:variance-bound} uses the Euclidean geometry of the original lattice. To distinguish the effect of the quadratic form from that of the kernel, we apply the second-moment comparison after a linear change of coordinates, using transformed unit cubes and the ball second moment in the centred model. The relation to the quadratic-form constant, together with Theorem~\ref{thm:quadratic-optimality} and Proposition~\ref{prop:variance-comparison}, yields the optimisation and comparison. The transformed-variance proof does not use the Gaussian probability estimate.

\subsection{Euclidean Prerequisites}

For a positive integer $k$ and $R\ge0$, write
\[
B_{k,2}(R)=\{x\in\mathbb R^k:\|x\|_2\le R\},\qquad
v_{k,2}=\operatorname{vol}(B_{k,2}(1)).
\]

\begin{proposition}\label{thm:pnorm-volume}
Let $k\ge1$ be an integer and let $R\ge0$. The volume of $B_{k,2}(R)$ is $v_{k,2}R^k$, where
\[
v_{k,2}=\frac{\pi^{k/2}}{\Gamma(1+k/2)}.
\]
Moreover,
\begin{equation}\label{thm:pnorm-surface}
\int_{B_{k,2}(R)}\|x\|_2^2\,dx=\frac{k}{k+2}v_{k,2}R^{k+2}.
\end{equation}
\end{proposition}

\begin{proof}
First let $v=\operatorname{vol}(B_{k,2}(1))$ be the unevaluated unit-ball volume. The inclusions
\[
[-k^{-1/2},k^{-1/2}]^k\subseteq B_{k,2}(1)\subseteq[-1,1]^k
\]
give $0<v<\infty$. Homogeneity of the norm gives $B_{k,2}(r)=rB_{k,2}(1)$ for $r>0$, so its volume is $vr^k$. At $r=0$, the ball is a singleton of volume zero since $k\ge1$.

Thus the image of Lebesgue measure under $x\mapsto\|x\|_2$ assigns mass $vr^k$ to $[0,r]$, and has density $kvr^{k-1}$ for $r>0$. For every nonnegative Borel function $f$, this gives the radial integration formula
\[
\int_{\mathbb R^k}f(\|x\|_2)\,dx
=\int_0^\infty f(r)\,d(vr^k)
=kv\int_0^\infty f(r)r^{k-1}\,dr.
\]

The substitution $s=x^2$ on the positive half-line gives
\[
\int_{\mathbb R}e^{-x^2}\,dx=\Gamma(1/2)=\sqrt\pi.
\]
Since the integrand below is a nonnegative product, Tonelli's theorem gives
\[
\int_{\mathbb R^k}e^{-\|x\|_2^2}\,dx=\pi^{k/2}.
\]
Applying the radial formula with $f(r)=e^{-r^2}$ and substituting $s=r^2$ gives
\[
kv\int_0^\infty e^{-r^2}r^{k-1}\,dr
=\frac{kv}{2}\Gamma(k/2)=v\Gamma(1+k/2).
\]
Since $k\ge1$, the gamma integrals are finite, and the two evaluations give
\[
v=\frac{\pi^{k/2}}{\Gamma(1+k/2)}=v_{k,2}.
\]
Finally, apply the radial formula with $f(r)=r^2\mathbf1_{[0,R]}(r)$, where $\mathbf1_E$ denotes the indicator of a set $E$:
\[
\int_{B_{k,2}(R)}\|x\|_2^2\,dx
=kv_{k,2}\int_0^R r^{k+1}\,dr
=\frac{k}{k+2}v_{k,2}R^{k+2}.
\]
This also holds at $R=0$, when both sides vanish.
\end{proof}

\begin{proposition}\label{prop:variance-comparison}
For every positive integer $k$,
\begin{equation}\label{eq:variance-constant-comparison}
D_{k,2}\le D_{1,2}=\frac1{\sqrt3}<C_1=\sqrt{\frac2\pi}\le C_k.
\end{equation}
\end{proposition}

\begin{proof}
Let $Q=[-1,1]^k$ and $B=B_{k,2}(\rho)$, where
\[
\rho=\left(\frac{2^k}{v_{k,2}}\right)^{1/k}
=\frac{2\Gamma(1+k/2)^{1/k}}{\sqrt\pi}.
\]
By Proposition~\ref{thm:pnorm-volume}, $\operatorname{vol}(B)=\operatorname{vol}(Q)=2^k$. The sets $Q\setminus B$ and $B\setminus Q$ have equal volume. The function $\|x\|_2^2$ is at least $\rho^2$ on the first set and at most $\rho^2$ on the second. Subtracting the two integrals, which are finite, therefore gives
\[
\int_B\|x\|_2^2\,dx\le\int_Q\|x\|_2^2\,dx.
\]
The ball moment formula \eqref{thm:pnorm-surface} and integration of each coordinate over the cube yield
\[
\frac{k}{k+2}\rho^2
\le\frac1{2^k}\int_Q\sum_{j=1}^k|x_j|^2\,dx
=\frac{k}{3}.
\]
Dividing by $k>0$ and taking positive square roots gives
\[
D_{k,2}=\frac{\rho}{\sqrt{k+2}}\le\frac1{\sqrt3}=D_{1,2}.
\]
Here $D_{k,2}$ is defined in \eqref{eq:variance-bound}, and $\Gamma(3/2)=\sqrt\pi/2$ gives the displayed scalar value. Since $\pi<4$, we have $1/\sqrt3<\sqrt{2/\pi}$. Finally, $b_{I_k}=k$ and $\det I_k=1$, so the definition of $C(A)$ gives $C(I_k)=\sqrt{2/\pi}$. Theorem~\ref{thm:quadratic-optimality} then gives $C_k\ge C(I_k)=C_1$, completing \eqref{eq:variance-constant-comparison}.
\end{proof}

\subsection{The Variance Bound on a Transformed Lattice}\label{app:variance}

For a real symmetric positive definite $k\times k$ matrix $A$, define
\[
V_k(A)=D_{k,2}\frac{\sqrt{k}(\det A)^{1/(2k)}}{\sqrt{b_A}}
=\frac{D_{k,2}}{C_1}C(A),\qquad V_k=V_k(H).
\]

\begin{proposition}\label{prop:ellipsoidal-variance}
Let $k,n\ge1$ be integers, let $\Sigma$ be admissible as in Definition~\ref{def:setting}, and let $A$ be a real symmetric positive definite $k\times k$ matrix. Then
\begin{equation}\label{eq:finite-ellipsoidal-variance}
\frac{M\sqrt n}{2^{n/k}}
\ge\left(\max\left\{V_k(A)^2-\frac{\operatorname{tr}A}{3b_A}2^{-2n/k},\,0\right\}\right)^{1/2}.
\end{equation}
For fixed $k,A$, this gives $M\ge(V_k(A)-o(1))2^{n/k}/\sqrt n$, with an error depending only on $k,A,n$, uniformly over admissible sequences. Moreover,
\begin{equation}\label{eq:optimised-variance}
\sup_{A\succ0}V_k(A)=V_k=D_{k,2}\frac{C_k}{C_1},
\qquad D_{k,2}\le V_k<C_k.
\end{equation}
The first inequality is an equality for $k=1$ and is strict for $k\ge2$.
\end{proposition}

\begin{proof}[Proof of Proposition~\ref{prop:ellipsoidal-variance}]
Let $s_1,\ldots,s_N$ be the values of $X$ from Subsection~\ref{sec:model}, where $N=2^n$. They are distinct points in one translate of $\mathbb Z^k$, each with probability $1/N$. Let $U$ be independent of $X$ and uniform on $[-1/2,1/2)^k$. The sets
\[
A^{1/2}\bigl(s_i+[-1/2,1/2)^k\bigr)
\]
are disjoint and each has volume $\sqrt{\det A}$. Hence $A^{1/2}(X+U)$ is uniform on their union $E$, of volume $N\sqrt{\det A}$.

Put $R=(N\sqrt{\det A}/v_{k,2})^{1/k}$. The sets $E\setminus B_{k,2}(R)$ and $B_{k,2}(R)\setminus E$ have equal volume. The function $\|x\|_2^2$ is at least $R^2$ on the first and at most $R^2$ on the second. All the integrals are finite because the sets are bounded. The ball moment formula in Proposition~\ref{thm:pnorm-volume} therefore gives
\[
\mathbb E\|A^{1/2}(X+U)\|_2^2
=\frac1{N\sqrt{\det A}}\int_E\|x\|_2^2\,dx
\ge\frac{k}{k+2}R^2.
\]
The cube averages satisfy $\mathbb EU=0$ and $\mathbb E[UU^{\mathsf T}]=I_k/12$. Independence of $X$ and $U$ eliminates the mixed term, so
\[
\mathbb E\|A^{1/2}(X+U)\|_2^2
=\mathbb E[X^{\mathsf T}AX]+\frac{\operatorname{tr}A}{12}.
\]
For the signs defining $X$, independence and symmetry give $\mathbb E[r_ir_j]=0$ for $i\ne j$ and $\mathbb E[r_i^2]=1$. Since $a_i/M\in[0,1]^k$,
\[
\mathbb E[X^{\mathsf T}AX]
=\frac14\sum_{i=1}^n a_i^{\mathsf T}Aa_i
\le\frac{nM^2b_A}{4}.
\]
Combining these estimates yields
\[
\frac{nM^2b_A}{4}+\frac{\operatorname{tr}A}{12}
\ge\frac{k}{k+2}\left(\frac{2^n\sqrt{\det A}}{v_{k,2}}\right)^{2/k}.
\]
Using $v_{k,2}=\pi^{k/2}/\Gamma(1+k/2)$, dividing by $2^{2n/k}b_A/4>0$, and then taking nonnegative square roots gives \eqref{eq:finite-ellipsoidal-variance}. For fixed $k,A$, its right side tends to $V_k(A)>0$ and depends only on $k,A,n$, proving uniformity. No invertibility of the covariance matrix of $X$ is required.

The identity $V_k(A)=D_{k,2}C(A)/C_1$ and Theorem~\ref{thm:quadratic-optimality} give the supremum in \eqref{eq:optimised-variance}. Proposition~\ref{prop:variance-comparison} gives $D_{k,2}<C_1\le C_k$, hence $D_{k,2}\le V_k<C_k$.

For $k=1$, $C_k=C_1$. For $k\ge2$, the second moment matrix $S$ in \eqref{eq:vertex-second-moment} has the distinct eigenvalues $b_H/k$ and $(k+1)b_H/k$. Thus the matrix arithmetic--geometric mean inequality in the proof of Theorem~\ref{thm:quadratic-optimality} is strict at $A=I_k$, giving $C_1<C_k$ and therefore $D_{k,2}<V_k$.
\end{proof}

The ratio $V_k/D_{k,2}=C_k/C_1$ records the improvement obtained by optimising the quadratic form in the variance argument. For the same quadratic form, the Gaussian constant exceeds the variance constant by the factor below, which equals $C_k/V_k$ at $A=H$:
\[
\frac{C(A)}{V_k(A)}=\frac{C_1}{D_{k,2}}>1.
\]

\section{Concluding Remarks}\label{sec:conclusion}

For fixed $k$, combining Theorem~\ref{mainthe} with the construction following Definition~\ref{def:setting} gives, as $n\to\infty$,
\[
\bigl(C_k-o(1)\bigr)\frac{2^{n/k}}{\sqrt n}\le m_k(n)\le2^{\lceil n/k\rceil-1}.
\]
These bounds do not determine the precise asymptotics of $m_k(n)$.

A possible refinement is to retain the covariance matrix rather than pass immediately to the trace bound. In the application of Lemma~\ref{lem:gaussian-probability}, for fixed $A\succ0$ we have
\[
z_i=A^{1/2}u_i,\qquad b=b_A,\qquad
T=\frac1n\sum_{i=1}^n z_i z_i^{\mathsf T}\succeq0.
\]
Equation~\eqref{eq:gaussian-comparison} retains the factor $\det(I_k+2tT)^{-1/2}$, with an error uniform over the sequence for fixed $k,A,t$, even when $T$ is singular. The present argument does not determine whether admissibility constrains $T$ or the subset-sum locations sufficiently to improve the uniform leading constant. Theorem~\ref{thm:quadratic-optimality} rules out such an improvement by optimising the formula $C(A)$ alone; it does not rule out sharper finite-$n$ errors or refinements before the trace estimate.

\appendix
\section{Geometric Interpretation of the Quadratic-Form Certificate}\label{app:geometry}

\begin{remark}[Geometric interpretation]\label{rem:ellipsoid-context}
Put $K=\operatorname{conv}([0,1]^k\cup[-1,0]^k)$. For every real symmetric positive definite $A$, convexity of the quadratic form shows that the ellipsoid
\[
E_A=\{x:x^{\mathsf T}Ax\le b_A\},\qquad
\operatorname{vol}(E_A)=\omega_k\frac{b_A^{k/2}}{\sqrt{\det A}},
\]
contains $K$, where $\omega_k$ is the Euclidean unit-ball volume. Maximising $C(A)$ is therefore equivalent to minimising the volume of a centred ellipsoid containing $K$: any such ellipsoid can be tightened to one of the form $E_A$. Other centres cannot lower the minimum. Indeed, if $(x-c)^{\mathsf T}A(x-c)\le1$ contains the symmetric set $K$, averaging the inequalities for $x$ and $-x$ gives
\[
x^{\mathsf T}Ax\le1-c^{\mathsf T}Ac\qquad(x\in K).
\]
Since $K$ has nonempty interior, $1-c^{\mathsf T}Ac>0$. This defines a centred containing ellipsoid of no greater volume.

Let $\mathbf1_{k+1}$ be the all-ones vector in $\mathbb R^{k+1}$, let $P$ be the orthogonal projection onto $\mathbf1_{k+1}^{\perp}$, and put $Lx=P(x,0)$. Then
\[
(Lx)\cdot(Ly)=x\cdot y-\frac{(\sum_i x_i)(\sum_i y_i)}{k+1}
=x^{\mathsf T}Hy,
\]
so $L^{\mathsf T}L=H$ and $L$ is an isomorphism onto this hyperplane. For $a\in\{0,1\}^k$, $P(a,0)=La$ and $P(a,1)=L(a-\mathbf1)$. Taking convex hulls and using $P\mathbf1_{k+1}=0$ gives
\[
LK=P[0,1]^{k+1}=P[-1/2,1/2]^{k+1}.
\]
This projected cube is the closed Voronoi cell of the root lattice $A_k=\mathbb Z^{k+1}\cap\mathbf1_{k+1}^{\perp}$; see McKilliam, Subramanian, Viterbo, and Clarkson \cite[Theorem~1, preprint p.~5]{mckilliam2011error}, taking the closure of their cell. With the intrinsic Euclidean metric on the hyperplane, $L$ maps $E_H$ onto the Euclidean ball of radius $\sqrt{b_H}$. The proof of Theorem~\ref{thm:quadratic-optimality} gives an explicit second-moment certificate for its minimum-volume property. The Voronoi-cell identification is classical geometric context, not an input to the Gaussian bound.
\end{remark}

\section*{Declaration of generative AI and AI-assisted technologies in the manuscript preparation process}
OpenAI ChatGPT and Codex were used for literature and bibliographic assistance, manuscript editing and organisation.

The author independently verified the mathematical arguments and takes full responsibility for the content of the manuscript.

\bibliographystyle{plain}
\bibliography{ERDOS-references}
\end{document}